\documentclass[a4paper,12pt]{amsart}

\usepackage{a4wide}
\usepackage{pgf,tikz}
\usepackage{amssymb}
\allowdisplaybreaks
\usepackage{enumerate}
	
\usepackage{amsmath}

\newif\ifdetails
\detailstrue
\newcommand{\DETAIL}[1]%
{\ifdetails\par\fbox{\begin{minipage}{0.9\linewidth}\textit{Detail:}
      #1\end{minipage}}\par\fi}
\newcommand{\TODO}[1]%
{\ifdetails\par\fbox{\begin{minipage}{0.9\linewidth}\textbf{TODO:}
      #1\end{minipage}}\par\fi}

\usepackage{makecell}
\usepackage{relsize}

\usepackage[pagewise,displaymath, mathlines]{lineno}
\newtheorem{lemma}{Lemma}
\newtheorem{proposition}[lemma]{Proposition}
\newtheorem{theorem}[lemma]{Theorem}
\newtheorem{corollary}[lemma]{Corollary}
\theoremstyle{remark}

\newtheorem{conjecture}{Conjecture}

\usepackage{caption}
\usepackage{subcaption}
\usepackage{float} % provides H as float placement specifier
\usepackage[pdftex,a4paper,
citecolor = blue, colorlinks=true,urlcolor=blue]{hyperref}
\usepackage{mathrsfs}
\usetikzlibrary{arrows}

\usepackage{xcolor}

\newcommand{\old}[1]{{}}
\title{On the average size of $1$-nearly independent vertex sets in graphs}

\author{Audace A. V. Dossou-Olory}

\author{Eric O. Andriantiana}
\thanks{}

\address{Audace A. V. Dossou-Olory \\  Institut National de l'Eau \\ Centre d'Excellence d'Afrique pour l'Eau et l'Assainissement \\ and  Institut de Math\'ematiques et de Sciences Physiques, Dangbo \\ Universit\'e d'Abomey-Calavi, B\'enin \\ \textbf{ORCID} \url{0000-0003-2065-117X}}
\email{audace@aims.ac.za}

\address{Eric O. D. Andriantiana \\ Department of Mathematics (Pure and Applied) \\ Rhodes University \\ Makhanda, 6140 South Africa \\ and National Institute for Theoretical and Computational Sciences (NITheCS), Stellenbosch, South Africa}
\email{e.andriantiana@ru.ac.za}


\begin{document}

\begin{abstract}
A $k$-nearly independent vertex subset of a graph $G$ is a set of vertices that induces a subgraph containing exactly $k$ edges. For $k = 0$, this coincides with the classical notion of independent subsets. This paper investigates the average size, $av_1(G)$ of the $1$-nearly independent vertex subsets of both graphs and trees of a given order $n$.

Let $E_n$ denote the $n$-vertex edgeless graph, so that $av_1(E_n) = 0$. We determine all $n$-vertex graphs $G\neq E_n$ that minimize or maximize $av_1$. Similarly, we identify the trees of order $n$ that achieve the minimum value of $av_1$, and prove that the maximum value lies between $n/2$ and $(n+1)/2$ if $n>8$. Finally, we construct a family of $n$-vertex trees which shows that the bounds are asymptotically sharp.

%\[
%\frac{n}{2}\ \leq max\{\, av_1(T) : T \in \mathbb{T}_n \,\} \le \frac{n+1}{2},
%\]
%and that this upper bound is asymptotically sharp.

%Let $\mathbb{G}_n$ be the set of all $n$-vertex graphs, and let $\mathbb{T}_n$ denote the set of all trees of order $n$.

% A $k$-nearly independent vertex subset of a graph $G$ is a set of vertices that induces a subgraph containing exactly $k$ edges. For $k = 0$, this corresponds to the well-known independent subsets. This studies the average size $av_1(G)$ of the $1$-nearly independent subsets of graphs (resp. trees) of given order $n$.

% Let $E_n$ be the $n$-vertex edgeless graph, then $av_1(E_n)=0$. Let $\mathbb{G}_n$ be the set of all $n$-vertex graphs and $\mathbb{T}_n$ the set of all trees of order $n$. 

% We determine all elements of $\mathbb{G}_n$ with minimum $av_1$, as well as those with maximum value. Similarly, we also determine the element of  $\mathbb{T}_n$ with minimum $av_1$.  Finally we show that $\max\{av_1(T): \mathbb{T}_n\}\leq (n+1)/2$, and that this bound is asymptomatically sharp.
\end{abstract}

\maketitle

\section{Introduction}

Counting substructures in graphs is a central topic in combinatorics and graph theory. Among these, the enumeration of subgraphs---particularly induced subgraphs, spanning subgraphs, independent sets and their generalisations---plays a crucial role in understanding graph complexity and structure. Key questions include which $n$-vertex graphs minimise or maximise the number or average size of independent sets, under various structural constraints? 

An independent vertex set in a graph $G$ is a set of vertices, no two of which are adjacent. The number of independent vertex sets of $G$ is also known as the Fibonacci number or Merrifield-
Simmons index of $G$ in the mathematical literature. Their enumeration has deep ties to certain models in statistical physics via the hard-core model, where each independent set is assigned a weight based on its size (see~\cite{Martin2009}). 

Let $i(G)$ denote the number of independent vertex sets (including the empty set) in graph $G$. Among $n$-vertex graphs, the minimum is achieved by the complete graph $K_n$ with $i(K_n) = n+1$. For graphs with a fixed number of vertices and edges, a classic result of Kahn~\cite{kahn2001} shows that the number of independent sets is maximized by a disjoint union of complete bipartite graphs. Specifically, among $d$-regular bipartite graphs $G$ on $n$ vertices such that $2d | n$, the maximum value of $i(G)$ is achieved by the disjoint union of $n/2d$ copies of $K_{d,d}$, confirming a conjecture of Alon and Kahn. In general graphs, Zhao~\cite{zhao2010} extended Kahn's entropy method to show that among all graphs with maximum degree $\Delta$, the number of independent vertex sets is asymptotically maximized by a disjoint union of $K_{\Delta, \Delta}$. Among regular graphs, the minimum is less understood, but lower bounds have been given in terms of vertex degrees. There are many other work on extremising $i(G)$ when $G$ is restricted to various classes of $n$-vertex graphs such as graphs with a given minimum degree~\cite{Galvin,Gan}, or given degree sequence~\cite{ANDRIANTIANA2013724}.

A \emph{1-nearly independent vertex set} allows exactly one pair of adjacent vertices in the collection. This concept was introduced recently by Andriantiana and Shozi in \cite{shozi2024vertex}, where various of its properties are discussed and extremal $n$-vertex graphs and $n$-vertex trees are characterised.  %Let $\alpha_1(G)$ be the maximum size of a 1-nearly independent vertex set. They gave tight bounds on $\alpha_1(G)$ and characterise extremal graphs.

The analogous edge version~\cite{shozi2024edge} counts edge-subsets with one adjacent pair; the star minimises and the path maximises this number among $n$-vertex trees. These new generalisations open further directions for extremal and enumerative analyses.

Let $\mathcal{S}_0(G)$ denote the set of all independent vertex subsets of graph $G=(V,E)$. Define the average size of independent vertex sets of $G$ by
\[
\bar{s}(G): = \frac{1}{i(G)} \sum_{I \in \mathcal{S}_0(G)} |I|.
\]
Andriantiana et al.~\cite{andriantiana2018} showed that among all $n$-vertex graphs, $\bar{s}$ is maximised by the edgeless graph and minimised by $K_n$, while among trees, the path minimises $\bar{s}$ and the star maximises it. See also the old short paper~\cite{Nilli1988AverageIndependentSets} which studies the asymptotic behaviour of $\bar{s}(G)$ in terms of order and the chromatic number of $G$,  and the recent results~\cite{Davies2018AverageIndependentSets} giving a tight lower bound on the expected size of a uniformly choosen independent set in triangle-free graphs of bounded degree.

In this manuscript, we investigate the average size $av_1$ of $1$-nearly independent subsets in the class of graphs of fixed order $n$, and subsequently in the class of trees of order $n$. Both upper and lower bounds are established, along with characterizations of the graphs that attain these bounds.

The paper is structured as follows. 
In Section~\ref{Sect:Prelim}, we formally define $av_1$, compute explicit expressions of $av_1$ for graphs that have less complicated structures, and establish some technical recursive formulas. These formulas will later be needed in proving our main theorems.
Section~\ref{Sect: extr} is devoted to the study of average among $n$-vertex graphs. In particular, graphs with the smallest and largest $av_1$ are characterized.
In Section~\ref{sect:tree} we show that the star is the $n$-vertex tree with the smallest $av_1$. We also prove that the maximum value of $av_1$ is less than $(n+1)/2$ for $n>8$. Moreover, we characterise an explicit family of $n$-vertex trees $R_n$ with $av_1(R_n)=n/2 + \delta_n$ for some $0 <\delta_n <1/2$, thus matching our bounds.
Finally, we provide further relations between $0$- and $1$-nearly independent vertex sets in graphs in Section~\ref{Sect:Furth}.

\section{Preliminary results}\label{Sect:Prelim}

Throughout this paper, all graphs are simple. The number of elements of a finite set $S$ is denoted by $|S|$. The empty sum is treated as $0$.

Let $G=(V,E)$ be a simple graph and $l$ a non-negative integer. As defined in \cite{shozi2024vertex}, a set $S\subseteq V$ is called a $l$-nearly independent vertex set of $G$ if the graph induced by $S$ in $G$ contains exactly $l$ edges. Thus $0$-nearly independent vertex sets are precisely the classical independent vertex sets. Denote by $\mathcal{S}_l(G)$ the set of all $l$-nearly independent vertex sets of $G$, and set
\begin{align*}
\sigma_l(G):=|\mathcal{S}_l(G)| \quad \text{and} \quad S_l(G):=\sum_{S \in \mathcal{S}_l(G)} |S|\,,
\end{align*}
that is $\sigma_l(G)$ is the number of $l$-nearly independent vertex sets of $G$, while $S_l(G)$ is the sum of the sizes of the $l$-nearly independent vertex sets of $G$. We are specifically interested in the average size of a $l$-nearly independent vertex set of $G$, which is given by
$$av_l(G):=\frac{S_l(G)}{\sigma_l(G)}\,,$$
provided that $\sigma_l(G)\neq 0$. If  $\mathcal{S}_l(G)=\emptyset$, then it is consistent to write $av_l(G)=0$.

%\medskip
%For our purposes, we need to introduce some auxiliary quantities.

We denote by $\sigma_l(G,k)$ the number of $l$-nearly independent vertex sets of $G$ that have cardinality $k$. Clearly, $\sigma_0(G,0)=1$, while $\sigma_l(G,0)=0$ for all $l>0$ and $\sigma_1(G,k)=0$ for all $k<2$. We define the $l$-nearly series $I_l(G;x)$ by
$$ I_l(G;x)=\sum_{k\geq 0} \sigma_l(G,k) x^k\,,$$ which has a finite number of terms since $\sigma_l(G,k)=0$ for all $k> |V(G)|$. We refer to $I_l(G;x)$ as the $l$-nearly independent vertex polynomial of $G$. It follows that
\begin{align*}
\sigma_l(G)&=\sum_{k\geq 0} \sigma_l(G,k) = I_l(G;1)\quad \text{and}\\
S_l(G)&=\sum_{k\geq 0} k \cdot \sigma_l(G,k) = \frac{d}{dx} I_l(G;x)\Big|_{x=1}:=I_l'(G;1)\,.
\end{align*}

An edge $e$ with end vertices $u,v$ in a graph $G$ will be denoted by $e=uv$. Let $G_1$ and $G_2$ be two graphs such that their vertex sets are disjoints. Then $G_1 \cup G_2$ means the graph with vertex set $V(G_1) \cup V(G_2)$ and edge set $E(G_1) \cup E(G_2)$. In this case, if $S$ is a $1$-nearly independent vertex set of $G_1 \cup G_2$, then the unique edge $e=uv$ in the subgraph induced by $S$ necessarily comes from $G_1$ or $G_2$, and $S\backslash \{u,v\}$ is an independent vertex set of $G_1 \cup G_2$. Conversely, $S\cap V(G_j)$ is either an independent or a $1$-nearly independent vertex set of $G_j$ for $j\in \{1,2\}$, depending on the location of edge $e=uv$.

We therefore obtain the following identities.

\begin{lemma}
\label{Lem:1}
Let $G_1$ and $G_2$ be two vertex disjoint graphs. We have
\begin{align*}
 \sigma_1(G_1 \cup G_2,k)&= \sum_{l=2}^k \sigma_1(G_1,l) \sigma_0(G_2,k-l)  +\sigma_1(G_2,l) \sigma_0(G_1,k-l) \quad \text{and}\\
 \sigma_1(G_1 \cup G_2)&= \sigma_1(G_1) \sigma_0(G_2) + \sigma_1(G_2) \sigma_0(G_1)\,.
\end{align*}
\end{lemma}

\medskip 

We also obtain the following.
\begin{lemma}\label{G1UnionG2}
Let $G_1$ and $G_2$ be two vertex disjoint graphs. We have
\begin{align}
\label{Eq:Er1}
  I_1(G_1 \cup G_2;x)&=I_1(G_1;x)I_0(G_2;x)+I_1(G_2;x)I_0(G_1;x)\,\text{ and }\\
  \label{Eq:Er2}
 S_1(G_1 \cup G_2)&=S_1(G_1)\sigma_0(G_2)+\sigma_1(G_1)S_0(G_2) + S_1(G_2)\sigma_0(G_1)+\sigma_1(G_2)S_0(G_1)  \,.
\end{align}
\end{lemma}

\begin{proof}
Formula~\eqref{Eq:Er1} follows immediately from Lemma \ref{Lem:1}. For \eqref{Eq:Er2},
simply note that 
 \begin{align*}
  &I_1'(G_1 \cup G_2;x)=I_1'(G_1;x)I_0(G_2;x)+ I_1(G_1;x)I_0'(G_2;x) \\
  & \hspace{4cm} +I_1'(G_2;x)I_0(G_1;x) + I_1(G_2;x)I_0'(G_1;x) \,.
 \end{align*}
% %\begin{align*}
% % &I_1'(G_1 \cup G_2;x)=I_1'(G_1;x)I_0(G_2;x)+ I_1(G_1;x)I_0'(G_2;x) \\
% % &\makebox[\textwidth][r]{$+I_1'(G_2;x)I_0(G_1;x) + I_1(G_2;x)I_0'(G_1;x)$}\,.
% %\end{align*}
 \end{proof}

\medskip
Let $G=(V,E)$ be a graph. For $S\subseteq V$, we denote by $G-S$ the graph induced by $V\backslash S$ in $G$. For simplicity, we write $G-v$ instead of $G-\{v\}$. We use the standard notation that $N(v)$ is the set of all neighbours of $v\in V$ in $G$ and $N[v]=N(v)\cup \{v\}$.

We can count the contribution of $v$ to $\sigma_1(G,k)$ as follows. If $v$ is an element of $S\in \mathcal{S}_1(G)$, then $S$ can only contain at most one neighbour of $v$ in $G$. If $S$ contains both $v$ and $u\in N(v)$, then no other element of $N(v) \cup N(u)$ can belong to $S$ and $S\backslash \{u,v\}$ is an independent vertex set of $G$. Thus the contribution of $v$ to $\sigma_1(G,k)$ is given by
 \begin{align*}
 \sigma_1(G-N[v],k-1) + \sum_{u \in N(v)} \sigma_0(G-(N[v]\cup N[u]),k-2)\,
 \end{align*}
 and thus
  \begin{align*}
 \sigma_1(G,k)= \sigma_1(G-v,k)+ \sigma_1(G-N[v],k-1) + \sum_{u \in N(v)} \sigma_0(G-(N[v]\cup N[u]),k-2)\,.
 \end{align*}
Note that $G-(N[v]\cup N[u])$ and $G-(N(v)\cup N(u))$ are isomorphic graphs since $uv\in E$. 

\medskip

The next lemma follows immediately, and its proof is omitted. Note that the recursive formula for $S_1$ is obtained by differentiating the expression for $I_1$ and evaluating it at $x = 1$.
%Now we derive a recursive formula for $I_1(.;x), ~\sigma_1(.)$ and $S_1(.)$.

\begin{lemma}\label{lem:g-v}
Let $G$ be a graph and $v$ a vertex of $G$. We have
\begin{align*}
I_1(G;x)&=I_1(G-v;x)+xI_1(G-N[v];x)+x^2\sum_{u\in N(v)}I_0 \big( G-(N(v)\cup N(u)); x \big)\,,\\
\sigma_1(G)&=\sigma_1(G-v)+\sigma_1(G-N[v]) + \sum_{u\in N(v)}\sigma_0\big(G-(N(v)\cup N(u)) \big) \,, \quad \text{ and }\\
S_1(G)&=S_1(G-v)+\sigma_1(G-N[v]) + S_1(G-N[v])\\
 & ~ + \sum_{u\in N(v)}\Big( 2\sigma_0\big(G-(N(v)\cup N(u)) \big) + S_0\big(G-(N(v)\cup N(u))\big) \Big)\,.
\end{align*}
\end{lemma}

% \begin{proof}
% Differentiate $I_1(G;x)$ with respect to $x$.

% \end{proof}
\medskip
In what follows, we compute $\sigma_1(G), S_1(G)$ for a few classes of graphs. We use the following usual notation for $n$-vertex graphs:
\begin{itemize}
\item $E_n$: graph with no edges (edgeless graph),
\item $S_n$: star (a central vertex adjacent to $n-1$ pendent edges),
\item $K_n$: complete graph (every two vertices are adjacent),
\item $P_n$: path (a tree with maximum degree less than or equal to $2$).
%\item $C_n$: cycle.
\end{itemize}

\begin{proposition}\label{Prop:Few} 
\begin{align*}
&\sigma_1(E_n)=S_1(E_n)=0\,,\\
&\sigma_1(S_n)=n-1,~ S_1(S_n)=2(n-1)\,, \quad \text{ and }\\
&\sigma_1(K_n)=n(n-1)/2, ~S_1(K_n)=n(n-1).
\end{align*}

\end{proposition}

\begin{proof}
All these expressions are easily obtained by direct counting arguments. Nevertheless, let us use the above decomposition formulas. $E_n$ has no edge, so $\mathcal{S}_1(E_n)=\emptyset$. Let $v$ be the central vertex of $S_n$, we have $S_n-v=E_{n-1}$, $S_n-N(v)=\emptyset$, and
\begin{align*}
\sigma_1(S_n)&=\sigma_1(E_{n-1})+\sigma_1(\emptyset)+\sum_{u\in N(v)}\sigma_0(\emptyset) =|N(v)|=n-1\,,\\
S_1(S_n)&=S_1(E_{n-1})+\sigma_1(\emptyset)+ S_1(\emptyset)+\sum_{u\in N(v)}(2\sigma_0(\emptyset)+S_0(\emptyset)) =2|N(v)|=2(n-1)\,.
\end{align*}
For the complete graph,
\begin{align*}
\sigma_1(K_n)&=\sigma_1(K_{n-1})+\sigma_1(\emptyset) + \sum_{u\in N(v)}\sigma_0(\emptyset)=\sigma_1(K_{n-1})+ |N(v)|=\sigma_1(K_{n-1})+ n-1\,,\\
S_1(K_n)&=S_1(K_{n-1})+\sigma_1(\emptyset) + S_1(\emptyset) + \sum_{u\in N(v)}\big( 2\sigma_0 (\emptyset) + S_0(\emptyset) \big)\\
&= S_1(K_{n-1})+ 2|N(v)|=S_1(K_{n-1})+2(n-1)\,.
\end{align*}
Solving these recursions yield $\sigma_1(K_n)=n(n-1)/2$ and $S_1(K_n)=n(n-1)$.
\end{proof}

\medskip
Let $G=(V,E)$ be a graph. Since every element of $\mathcal{S}_1(G)$ induces a graph that contains only one edge, the set $\mathcal{S}_1(G)$ can be partitioned into 
\begin{align*}
\mathcal{S}_1(G)=\bigcup_{e\in E}~ \bigcup_{k\geq 2}~ \mathcal{S}^{e}_{1,k}(G)\,,
\end{align*}
where $\mathcal{S}^{e}_{1,k}(G)$ comprises all those $k$-element $1$-nearly independent vertex sets involving the end vertices of edge $e$. One can therefore determine $|\mathcal{S}^{e}_{1,k}(G)|$, the contribution of edge $e=uv$ to $\sigma_1(G,k)$. Following the same argument given above, we get
\begin{align*}
|\mathcal{S}^{e}_{1,k}(G)|&=\sigma_0\big(G-(N(v) \cup N(u)), k-2 \big)\,,\\
\sigma_1(G,k)&=\sum_{uv\in E} \sigma_0\big(G-(N(v) \cup N(u)), k-2 \big)\,.
\end{align*}

As an immediate consequence, we obtain the following formulas:
\begin{lemma}\label{av1Exp}
For every graph $G$, it holds that
\begin{align*}
I_1(G;x)&=x^2 \sum_{uv \in E} I_0(G-(N(v) \cup N(u)); x)\,,\\
\sigma_1(G)&=\sum_{uv \in E} \sigma_0\big(G-(N(v) \cup N(u))\big)\,,\\
S_1(G)&=\sum_{uv \in E} \Big(2\sigma_0\big(G-(N(v) \cup N(u))\big) +S_0\big(G-(N(v) \cup N(u))\big) \Big)\\
&=\sum_{uv \in E} \sigma_0\big(G-(N(v) \cup N(u))\big)\Big(2 +av_0\big(G-(N(v) \cup N(u))\big) \Big)\,.
\end{align*}
\end{lemma}

\medskip
For example, setting $E(P_n):=\{v_0v_1, v_1v_2, \ldots, v_{n-2}v_{n-1}\}$ and using this lemma we can compute $\sigma_1(P_n)$ and $S_1(P_n)$ as
\begin{align*}
\sigma_1(P_n)=\sum_{j=1}^{n-1} \sigma_0(P_{j-2} \cup P_{n-2-j} )=\sum_{j=1}^{n-1} \sigma_0(P_{j-2} ) \sigma_0( P_{n-2-j})\,,
\end{align*}
where the second equality comes from Lemma~\ref{0G1UnionG2} below.

\begin{lemma}[\cite{andriantiana2018}]\label{0G1UnionG2}
Let $G_1$ and $G_2$ be two vertex disjoint graphs. We have
\begin{align*}
\sigma_0(G_1 \cup G_2)&= \sigma_0(G_1) \sigma_0(G_2)\,,\\
av_0(G_1 \cup G_2)&= av_0(G_1)+ av_0(G_2)\,.
\end{align*}
\end{lemma}

\medskip

The next lemma provides an alternative formula for $av_1(G)$ that will be useful in some cases.
%The next lemma is not necessarily needed to prove the minimality of the star. However, it will play a crucial role in further arguments. So we state it now.

\begin{lemma}\label{av_1Instance}
If $G$ is not an edgless graph, then
\begin{align*}
av_1(G)=\frac{S_1(G)}{\sigma_1(G)}=2+ \frac{\displaystyle\sum_{uv \in E} S_0\big(G-(N(v) \cup N(u))\big)}{\displaystyle\sum_{uv \in E} \sigma_0\big(G-(N(v) \cup N(u))\big)}\,.
\end{align*}
Furthermore, there exists edges $u_1v_1, u_2v_2$ of $G$ such that
\begin{align*}
2+  av_0\big(G-(N(v_1) \cup N(u_1))\big)  \leq av_1(G) \leq 2+ av_0\big(G-(N(v_1) \cup N(u_1))\big)\,.
\end{align*}
\end{lemma}

\begin{proof}
According to Lemma~\ref{av1Exp},
\begin{align*}
av_1(G)=2+ \frac{\displaystyle\sum_{uv \in E} S_0\big(G-(N(v) \cup N(u))\big)}{\displaystyle\sum_{uv \in E} \sigma_0\big(G-(N(v) \cup N(u))\big)}\,.
\end{align*}
Now set $H_{uv}:=G-(N(v) \cup N(u))$ and note that
\begin{align*}
av_1(G)=2+ \frac{\displaystyle\sum_{uv  E} S_0(H_{uv})}{\displaystyle\sum_{uv \in E} \sigma_0 (H_{uv})}=2+\sum_{uv \in E}\alpha_{uv}\frac{S_0(H_{uv})}{\sigma_0(H_{uv})}=2+ \sum_{uv \in E} \alpha_{uv}av_0(H_{uv})\,,
\end{align*}
where $\alpha_{uv}=\sigma_0(H_{uv}) / \sum_{uv \in E} \sigma_0(H_{uv})$. Since $\alpha_{uv}>0$ and $\sum_{uv \in E} \alpha_{uv}= 1$, we can interpret $\sum_{uv \in E}\alpha_{uv}\frac{S_0(H_{uv})}{\sigma_0(H_{uv})}$ as an average. Hence, there is a case of $uv$ that corresponds to a term not less than the average, and there is a case that corresponds to a term not larger than the average. The lemma follows.
\end{proof}

\medskip
%Any other intermediate results will be covered in the next sections.

In the next section, we consider graphs with a given order $n$.

% It is known (see~\cite{andriantiana2018}) that
% \begin{align*}
% \sigma_0(P_m)&=\frac{1}{\sqrt{5}} (\alpha^{m+2} - \beta^{m+2})\quad \text{with} \quad \alpha=\frac{1+\sqrt{5}}{2}\,,~ \beta=\frac{1-\sqrt{5}}{2}\,,\\
% av_0(P_m)&=\frac{3-\sqrt{5}}{5}+ \frac{5-\sqrt{5}}{10}m - \frac{m+2}{\sqrt{5}((-\alpha^2)^{m+2} -1)}\,.
% \end{align*}

% Calculations show (see~{\color{red} REF}) that
% \begin{align*}
% \sigma_1(P_n)=\frac{1}{5}\Big((n-1) (\alpha^n + \beta^n) +\frac{2}{\sqrt{5}}  (\alpha^{n-1} - \beta^{n-1}) \Big)\,.
% \end{align*}

% Similarly, 
% \begin{align*}
% S_1(P_n)= \sum_{j=1}^{n-1}\sigma_0(P_{j-2})\sigma_0(P_{n-2-j})  (2+av_0(P_{j-2})+av_0(P_{n-2-j}) )\,.
% \end{align*}

%\section{$n$-vertex graphs with smallest $av_1$}

\section{Extremal $n$-vertex graphs} \label{Sect: extr}

In this section, we characterise the structure of graphs that minimise or maximise $av_1$.  Trivially, $av_1(G)\geq 0$ with equality if and only if $G$ is an edgeless graph. Here and in the following, we assume $G\neq E_n$ for all $n$.

In~\cite{andriantiana2018} it was proved that the star maximises $av_0$ among all $n$-vertex trees, while~\cite{shozi2024vertex} shows that the star minimises $\sigma_1$ among all $n$-vertex connected graphs. It will be shown that the star also uniquely achieves the minimum value of $av_1$ among connected graphs.  

\medskip
In paper~\cite{shozi2024vertex}, a {\it good graph} $G=(V,E)$ is defined as a graph in which every edge $uv$ satisfies $N(v) \cup N(u)=V$. A full characterisation of the set $\mathcal{H}$ of all good graphs is provided there. For any $H\in\mathcal{H}\smallsetminus\{K_1\}$, we have $av_1(H)=2$. Hence we have the following first main theorem of this paper.

\begin{theorem}\label{Thm:ngraph}
For every $n$-vertex graph $G$ that is not edgeless, we have $$av_1(G)\geq  2\,.$$ Equality holds if and only if $G \in \mathcal{H}\backslash \{ K_1 \}$.

In particular, $S_n$ uniquely minimises $av_1$ among $n$-vertex trees, and it is one of the $n$-vertex graphs that minimises $av_1$.
\end{theorem}

\begin{proof}
By Lemma~\ref{av_1Instance}, $av_1(G)\geq 2$ for all $G=(V,E)$ that are not edgeless, and equality holds if and only if $S_0\big(G-(N(v) \cup N(u))\big)=0$ for all $uv\in E$, i.e. every $G-(N(v) \cup N(u))$ is the empty graph, or equivalently $N(v) \cup N(u)=V$.

The only edgeless graph from $\mathcal{H}$ is $K_1$. %(see~{\color{red} REF}).
\end{proof}

\medskip
We use Lemma~\ref{av_1Instance} to obtain some bounds on $av_1$. For a graph $G=(V,E)$, define
\begin{align*}
\delta_1(G)=\min_{uv\in E}|N(v) \cup N(u)|\quad \text{and} \quad \delta_2(G)=\max_{uv\in E}|N(v) \cup N(u)|\,.
\end{align*}
It is clear that $2\leq \delta_1(G)\leq \delta_2(G)\leq |V|$ for all $G$ that are not edgless.

\begin{proposition}\label{prop:delta}
For every $n$-vertex graph $G\neq E_n$, 
\begin{align*}
2\leq \frac{3n+2-3\delta_2(G)}{n+1-\delta_2(G)} \leq av_1(G) \leq \frac{n+4- \delta_1(G)}{2} \leq \frac{n+ 2}{2}\,.
\end{align*}
\end{proposition}
 
\begin{proof}
Using $u_1,v_1,u_2$ and $v_2$ as in Lemma~\ref{av_1Instance} and the result 
\begin{align*}
\frac{n}{n+1}=av_0(K_n)\leq av_0(H)\leq av_0(E_n)=\frac{n}{2}
\end{align*}
that holds for all $n$-vertex graph $H$ (see~\cite{andriantiana2018}), we have
\begin{align*}
2+ \frac{n-|N(v_1) \cup N(u_1)| }{n+1-|N(v_1) \cup N(u_1)| } \leq av_1(G) \leq 2+\frac{n-|N(v_2) \cup N(u_2)| }{2}\,.
\end{align*}
The function $x \mapsto (a-x)/(b-x)$ is decreasing in $x$, provided that $a<b$. It follows that
\begin{align*}
2+ \frac{n-\delta_2(G)}{n+1-\delta_2(G)} \leq av_1(G) \leq 2+\frac{n- \delta_1(G)}{2} \leq 2+\frac{n- 2}{2}\,.
\end{align*}
\end{proof}

\medskip
We now aim to obtain a sharp upper bound on $av_1$.

\begin{proposition}[\cite{andriantiana2018}]
\label{Pro:avo}
For every $n$-vertex graph $G$ that is not edgeless, we have
$$av_0(G) < av_0(E_n) = n/2.$$
\end{proposition}

\medskip
Let $av_1(G,e)$ denote the average of all $1$-nearly independent subsets of $G$ that contain the endvertices of edge $e$. For $e=uv$, we have
$$
av_1(G,e)=2+ av_0\big(G-(N(u) \cup N(v)) \big)\,.
$$ %see the discussion before Lemma~\ref{av1Exp} in Section~\ref{Sect:Prelim}. 
Moreover, 
\begin{align*}
av_1(G)=2+ \sum_{uv \in E} \alpha_{uv} \cdot av_0 \big( G-(N(v) \cup N(u)) \big)
\end{align*}
for certain $\alpha_{uv}>0$ such that $\sum_{uv \in E} \alpha_{uv}= 1$, see the proof of Lemma~\ref{av_1Instance}. Since $av_0(E_m)$ increases with $m$, the above identities imply that
$$
av_1(G)\leq \max_{e\in E(G)} av_1(G,e) \leq 2+ av_0(E_{n-2})
$$ holds for every graph $G$ with $n$ vertices.

\medskip
Define $G_n:=K_2\cup (n-2)K_1$ for $n\geq 6$. This graph $G_n$ contains only one edge, thus
$$
av_1(G_n)=2+ av_0(E_{n-2})\,.
$$
Hence, our next theorem follows:

\begin{theorem}
\label{Th:GenMax}
Let $n\geq 6$. For every $n$-vertex graph $G$, 
\begin{align*}
av_1(G)\leq \frac{n}{2}+1\,.
\end{align*}
Equality holds if and only if $G=G_n$.
\end{theorem}

\medskip
We now shift our focus to the family of $n$-vertex trees.

\section{Extremal $n$-vertex trees}\label{sect:tree}

As a particular case of Theorem~\ref{Thm:ngraph}, the star $S_n$ uniquely minimises $av_1$ among $n$-vertex trees. The next task is to search for $n$-vertex trees with the maximum average size of $1$-nearly independent vertex sets. 

We call an internal vertex in a tree $T$ and vertex with degre greater than $1$.

\begin{proposition}
For any $n$-vertex tree $T$ with minimum degree  of internal vertices equal to $\delta'$, we have
$$
av_1(T)\leq 2+\frac{n-\delta'-1}{2}\,.
$$
\end{proposition}
\begin{proof}
As already seen in the proof of Lemma \ref{av_1Instance},
$$
av_1(T)=2+ \sum_{uv \in E(T)} \alpha_{uv} \cdot av_0\big( T-(N(u) \cup N(v)) \big)
$$
for some positive numbers $\alpha_{uv}$ that sum to $1$. If $T$ has $n$ vertices and minimum internal degree $\delta'$, then by Proposition \ref{Pro:avo},
$$av_0\big( T-(N(u) \cup N(v)) \big) \leq \frac{\max\{n-\delta'-1,0\}}{2}$$ for every $uv \in E(T)$. 
\end{proof}

\medskip
We can then deduce the following upper bound for $n$-vertex trees, obtained for $\delta'\in \{1,2\}$. The case of equality for $T=P_i$ with $2\leq i\leq 4$, can be checked with easy calculations.
\begin{corollary}
For any $n$-vertex tree $T$, we have
$$
av_1(T)\leq 2+\frac{\max\{n-3,0\}}{2}
$$
with equality holding for $T\in \{P_2, P_3, P_4\}$.
\end{corollary}

\medskip
We will prove a matching lower bound on the maximum average size of $1$-nearly independent vertex sets in trees.

For $n\geq 4$, define $R_n$ to be the tree obtained by subdividing once an edge of the star $S_{n-1}$, i.e. by adding an edge between a leaf of $S_{n-1}$ and a new vertex. 

Let $v$ be a leaf of $R_n$ such that $v$ is adjacent to a vertex $u$ of degree $2$. With this choice of vertex $v$, we apply Lemma~\ref{lem:g-v} alongside Proposition~\ref{Prop:Few} to obtain

\begin{align*}
av_1(R_n) &=\frac{S_1(S_{n-1})+S_1(S_{n-2}) +\sigma_1(S_{n-2}) + S_0(E_{n-3})+2\sigma_0(E_{n-3})}{\sigma_1(S_{n-1})+\sigma_1(S_{n-2})+\sigma_0(E_{n-3})}\nonumber\\
&=\frac{2(n-2)+2(n-3) + (n-3) + (n-3)2^{n-4}+2\times2^{n-3}}{(n-2)+(n-3)+2^{n-3}}\,,
\end{align*}
where in the last step we used the formulas
\begin{align*}
\sigma_0(E_m)=2^m\,, \quad  S_0(E_m)=m 2^{m-1}\,.
\end{align*}
Thus
\begin{align*}
av_1(R_n) &=\frac{ 5n-13+ (n+1)2^{n-4}}{2n-5+2^{n-3}}\nonumber\\
&=\frac{n}{2}+\frac{2^{n-4}- n^2+15n/2-13}{2n-5+2^{n-3}}\nonumber\,,
%\label{Eq:Minfn}
\end{align*}
and we also have
\begin{align}
av_1(R_n)&=\frac{ 5n-13+ (n+1)2^{n-4}}{2n-5+2^{n-3}}\nonumber\\
&= \frac{n+1}{2} - \frac{n^2 -13n/2 +21/2}{2n-5+2^{n-3}} <\frac{n+1}{2}\,.
\label{Eq:n1}
\end{align} 

%Define $f(n)=(5n-13+ (n+1)2^{n-4})/(2n-5+2^{n-3})$. Then
%\begin{align}
%\label{Eq:fnBound}
%\frac{n}{2}<f(n)<\frac{n+1}{2}
%\end{align}
%for {\color{blue} $n\geq 4$}, and the derivative of $f(n)$  is 
%\begin{align*}
%f'(n)
%&=\frac{2^{n + 4}n^2\log(2) - 13\cdot2^{n + 3}n\log(2) + 21\cdot2^{n + 3}\log(2) + 2^{2n} + 3\cdot2^{n + 3} + 128}{2(2^n + 16n - 40)^2 }>0
%\end{align*}
%for $n\geq 4.$ So, $f(n)$ is an increasing function of $n$.

%Define $g(n)=\frac{-2(n-3)^2+6}{2n-5+2^{n-3}}$.

%Then $av_1(P_1)=0<2=f(1), av_1(P_2)=2$
Note that 
$$
\lim_{n\to \infty}av_1(R_n)-\frac{n+1}{2} =0
$$ 
and that
\begin{align}
av_1(R_n) &=\frac{n}{2}+\frac{2^{n-4}- n^2+15n/2-13}{2n-5+2^{n-3}}\nonumber > \frac{n}{2}\,,
\label{Eq:Minfn}
\end{align}
for all $n$, except $n\in \{6,7,8\}$.

\begin{theorem}
Let $n>8$. Among $n$-vertex trees, the maximum average size of $1$-nearly independent vertex sets is given by $n/2 + \delta_n$ for some function $0< \delta_n <1/2$. Moreover, both the upper and lower bounds are asymptotically sharp.
\end{theorem}

\medskip
At this stage, it is reasonable to make the following conjecture.

\begin{conjecture}
The tree $R_n$ reaches the maximum value of $av_1$ over the set of all $n$-vertex trees.
\end{conjecture}

\medskip
We conclude this work with further bounds relating $\sigma_0$ and $\sigma_1$.

\section{Further bounds} \label{Sect:Furth}

In this short section, we establish further identities and bounds. In particular, we show that it is possible to have $\sigma_0 \leq 3 \sigma_1$. We begin with a lemma.

\begin{lemma}\label{lem:help}
For every graph $G=(V,E)$ that is not edgeless and every $uv \in E$, we have 
\begin{align*}
\frac{1}{2^{l_{uv}-2}+2^{l_{uv}-m_u}+2^{l_{uv}-m_v}+1}\leq \frac{\sigma_0(G-(N[v]\cup N[u]))}{\sigma_0(G)}\leq 1- \frac{\sigma_0(G-v)}{\sigma_0(G)}\,,
\end{align*}
where
\begin{align*}
l_{uv}=|N(v)\cup N(u)|\,, \quad m_u=|N[u]|\,, \quad m_v=| N[v]|\,.
\end{align*}
In particular, 
\begin{align*}
\frac{\sigma_0(G-(N[v]\cup N[u])}{\sigma_0(G)}\geq \frac{1}{(3\cdot 2^{l_{uv}-2}+1)}\,.
\end{align*}

\end{lemma}

\begin{proof}
We return to the recursion
\begin{align*}
\sigma_0(G)=\sigma_0(G-v)+\sigma_0(G-N[v])\,.
\end{align*}
Since $G-(N[v]\cup N[u])$ is an induced subgraph of $G-N[v]$, we have 
\begin{align*}
\sigma_0(G-N[v]) &\geq \sigma_0(G-(N[v]\cup N[u])) \,,\\
1&\geq \frac{\sigma_0(G-v)}{\sigma_0(G)}+\frac{\sigma_0(G-(N[v]\cup N[u]))}{\sigma_0(G)}\,,
\end{align*}
which proves the second inequality. On the other hand, we iterate the inequality
\begin{align*}
\sigma_0(G)=\sigma_0(G-w)+\sigma_0(G-N[w])\leq 2\sigma_0(G-w)\,,
\end{align*}
which holds for all $w\in V$, to obtain
\begin{align*}
\sigma_0((G-v)-u)\leq &2^{|(N[v]\cup N[u])\backslash \{u,v\}|} \sigma_0\big(G-(N[v]\cup N[u]) \big)\\
 &=2^{l_{uv}-2} \sigma_0\big(G-(N[v]\cup N[u]) \big)\,,\\
\sigma_0((G-v)-N[u])\leq & 2^{|(N[v]\cup N[u]) \backslash (\{v\} \cup N[u])|} \sigma_0 \big(G-(N[v]\cup N[u]) \big)\\
&= 2^{l_{uv}-m_u} \sigma_0 \big(G-(N[v]\cup N[u]) \big)
\end{align*}
and
\begin{align*}
\sigma_0((G-N[v])-u)\leq & 2^{|(N[v]\cup N[u]) \backslash (\{u\} \cup N[v])|} \sigma_0 \big(G-(N[v]\cup N[u]) \big)\\
 &= 2^{l_{uv}-m_v} \sigma_0 \big(G-(N[v]\cup N[u]) \big)\,.
\end{align*}
It follows that
\begin{align*}
\sigma_0(G)&=\sigma_0(G-v)+\sigma_0(G-N[v])\\
&= \sigma_0((G-v)-u)+\sigma_0((G-v)-N[u])\\
& ~ ~ + \sigma_0((G-N[v])-u)+\sigma_0((G-N[v])-N[u])\\
&\leq (2^{l_{uv}-2}+2^{l_{uv}-m_u} +2^{l_{uv}-m_v}+1)\sigma_0(G-(N[v]\cup N[u])\,,
\end{align*}
which proves the first inequality in the lemma. In particular, 
\begin{align*}
\sigma_0(G)\leq & (2^{l_{uv}-2}+2^{l_{uv}-2} +2^{l_{uv}-2}+1)\sigma_0(G-(N[v]\cup N[u])\\
&=(3\cdot 2^{l_{uv}-2}+1)\sigma_0(G-(N[v]\cup N[u])\,,
\end{align*}
proving the last inequality.

\end{proof}

\medskip
In the next proposition, we show that $\sigma_1 /\sigma_0$ is a somewhat additive function in the union of graphs, and that $\sigma_0 \leq 3 \sigma_1$ can possibly hold in certain graph classes.

\begin{proposition}\label{prop:1sur3}
Let $G_1$ and $G_2$ be two vertex disjoint graphs. We have
\begin{align*}
\frac{ \sigma_1(G_1 \cup G_2)}{ \sigma_0(G_1 \cup G_2)}=\frac{ \sigma_1(G_1)}{ \sigma_0(G_1)}+\frac{ \sigma_1(G_2)}{ \sigma_0(G_2)}\,.
\end{align*}
Moreover, for every graph $G=(V,E)$ with no isolated vertex and maximum degree $2$,
\begin{align*}
\frac{\sigma_1(G)}{\sigma_0(G)}\geq \frac{1}{3}\,.
\end{align*}
Equality holds for $P_2$.
\end{proposition}

\begin{proof}
The first identity immediately follows from Lemma~\ref{G1UnionG2}:
\begin{align*}
\frac{ \sigma_1(G_1 \cup G_2)}{ \sigma_0(G_1 \cup G_2)}=\frac{\sigma_1(G_1) \sigma_0(G_2) + \sigma_1(G_2) \sigma_0(G_1)}{\sigma_0(G_1) \sigma_0(G_2)}\,.
\end{align*}
Now let $G=(V,E)$ be a graph with no isolated vertex and maximum degre $2$. All the connected components of $G$ are paths and/or cycles. If $P_2$ is such a component, then set $G_1:=P_2$ and $G_2:=G\backslash V(G_1)$ to obtain
\begin{align*}
\frac{ \sigma_1(G)}{ \sigma_0(G)}=\frac{ \sigma_1(P_2)}{ \sigma_0(P_2)}+\frac{ \sigma_1(G_2)}{ \sigma_0(G_2)}\geq \frac{ \sigma_1(P_2)}{ \sigma_0(P_2)}=\frac{1}{3}\,.
\end{align*}
Otherwise, 
\begin{align*}
l_{uv}:=|N(v)\cup N(u)|\in \{3,4\}
\end{align*}
for all $uv\in E$. Here we distinguish two cases:
\begin{itemize}
\item Case~1: $|E|\geq 5$.\\
We can partition $E$ as $E=E_4 \cup E_3$, where
\begin{align*}
E_4=\{uv\in E:~ l_{uv}=4 \}\quad \text{and} \quad E_3=\{uv\in E:~ l_{uv}=3 \}\,.
\end{align*}
We combine (see Lemma~\ref{lem:help})
\begin{align*}
\frac{\sigma_0(G-(N[v]\cup N[u])}{\sigma_0(G)}\geq \frac{1}{(3\cdot 2^{l_{uv}-2}+1)}\,.
\end{align*}
and (see~Lemma~\ref{av1Exp})
\begin{align*}
\frac{\sigma_1(G)}{\sigma_0(G)}=\sum_{uv \in E} \frac{\sigma_0\big(G-(N(v) \cup N(u))\big)}{\sigma_0(G)}
\end{align*}
to obtain
\begin{align*}
\frac{\sigma_1(G)}{\sigma_0(G)}&\geq \sum_{uv \in E_4} \frac{1}{(3\cdot 2^{4-2}+1)} ~ + ~\sum_{uv \in E_3} \frac{1}{(3\cdot 2^{3-2}+1)}=\frac{|E_4|}{13}+\frac{|E_3|}{7}\\
& \geq \frac{5-|E_3|}{13}+\frac{|E_3|}{7}=\frac{35+6|E_3|}{91}> \frac{1}{3}\,.
\end{align*}
\item Case~2: $1\leq |E|\leq 4$.\\
Here all the connected components of $G$ must belong to the set
\begin{align*}
Q:=\{P_5,C_4,P_4,C_3,P_3\}\,,
\end{align*}
while easy calculations yield
\begin{align*}
&\frac{\sigma_1(P_5)}{\sigma_0(P_5)}=\frac{10}{13}\,, \quad \frac{\sigma_1(C_4)}{\sigma_0(C_4)}=\frac{4}{7}\,, \quad \frac{\sigma_1(P_4)}{\sigma_0(P_4)}=\frac{5}{8}\,, \\
&\frac{\sigma_1(C_3)}{\sigma_0(C_3)}=\frac{3}{4}\,, \quad \frac{\sigma_1(P_3)}{\sigma_0(P_3)}=\frac{2}{5}\,.
\end{align*}
All these values are greater than $1/3$. Now $G=G_1 \cup G_2$ for some $G_1\in Q$. We have
\begin{align*}
\frac{ \sigma_1(G_1 \cup G_2)}{ \sigma_0(G_1 \cup G_2)}=\frac{ \sigma_1(G_1)}{ \sigma_0(G_1)}+\frac{ \sigma_1(G_2)}{\sigma_0(G_2)} > \frac{1}{3}\,,
\end{align*}
which concludes the proof.
\end{itemize}

\end{proof}

% \medskip
% In the next lemma, we establish a somewhat additive function on $av_1$.

% \begin{lemma}\label{lem:av1Autr}
% Let $G_1$ and $G_2$ be two vertex disjoint graphs, where both have non-empty edge sets. It holds that
% \begin{align*}
% av_1(G_1 \cup G_2)&= av_1(G_1)+av_1(G_2)\\
% &~ +\frac{(av_0(G_1)-av_1(G_1))\frac{\sigma_0(G_1)}{\sigma_1(G_1)} + (av_0(G_2)-av_1(G_2))\frac{\sigma_0(G_2)}{\sigma_1(G_2)}}{\frac{\sigma_0(G_1)}{\sigma_1(G_1)} + \frac{\sigma_0(G_2)}{\sigma_1(G_2)}}\,.
% \end{align*}
% %&= \frac{\sigma_0(G_2)}{\sigma_1(G_2)} av_1(G_1) + \frac{\sigma_0(G_1)}{\sigma_1(G_1)} av_1(G_2) + \frac{av_0(G_1) \frac{\sigma_0(G_1)}{\sigma_1(G_1)} + av_0(G_2)\frac{\sigma_0(G_2)}{\sigma_1(G_2)}}{\frac{\sigma_0(G_1)}{\sigma_1(G_1)} + \frac{\sigma_0(G_2)}{\sigma_1(G_2)}}
% %\,.

% \end{lemma}

% \begin{proof}
% A simple manipulation of the identities
% \begin{align*}
% \sigma_1(G_1 \cup G_2)&= \sigma_1(G_1) \sigma_0(G_2) + \sigma_1(G_2) \sigma_0(G_1)\,,\\
%  S_1(G_1 \cup G_2)&=S_1(G_1)\sigma_0(G_2)+\sigma_1(G_1)S_0(G_2) + S_1(G_2)\sigma_0(G_1)+\sigma_1(G_2)S_0(G_1)  
% \end{align*}
% yields the desired result.

% \end{proof}

\bibliographystyle{abbrv} 
\bibliography{Bib}

\end{document}